\documentclass[11pt]{amsart}

\usepackage{enumerate,graphicx}
\usepackage[latin1]{inputenc}

\usepackage{amsmath,amsthm,amssymb}
\usepackage{paralist,mathtools}
\usepackage{esint}

\usepackage[arrow, matrix, curve]{xy}

\usepackage{eucal,mathrsfs,yhmath}

\DeclareFontFamily{U}{matha}{\hyphenchar\font45}
\DeclareFontShape{U}{matha}{m}{n}{
      <5> <6> <7> <8> <9> <10> gen * matha
      <10.95> matha10 <12> <14.4> <17.28> <20.74> <24.88> matha12
      }{}
\DeclareSymbolFont{matha}{U}{matha}{m}{n}
\DeclareFontFamily{U}{mathx}{\hyphenchar\font45}
\DeclareFontShape{U}{mathx}{m}{n}{
      <5> <6> <7> <8> <9> <10>
      <10.95> <12> <14.4> <17.28> <20.74> <24.88>
      mathx10
      }{}
\DeclareSymbolFont{mathx}{U}{mathx}{m}{n}

\DeclareMathSymbol{\obot}         {2}{matha}{"6B}
\DeclareMathSymbol{\bigobot}       {1}{mathx}{"CB}

\renewcommand{\eqref}[1]{Equation~(\ref{#1})}

\newcommand{\NN}{\mathbb N}

\newcommand{\RR}{\mathbb R}
\newcommand{\CC}{\mathbb C}
\newcommand{\CP}{\CC\mathrm P}

\newcommand{\Hil}{\mathscr H}
\newcommand{\T}{T}
\newcommand{\U}{\mathrm U}
\newcommand{\N}{N}
\renewcommand{\P}{\mathrm P}

\newcommand{\vol}{\mathrm{vol}}

\newcommand{\pr}{\mathrm{pr}}

\newcommand{\real}{\mathrm{Re}}
\newcommand{\imaginary}{\mathrm{Im}}

\newcommand{\zcurr}{\zeta}
\newcommand{\icurr}{\sigma}
\newcommand{\pcurr}{\chi}

\newcommand{\subE}{{\!\!\raisebox{-.13em}{$\scriptscriptstyle E$}}}
\newcommand{\subH}{{\!\!\raisebox{-.13em}{$\scriptscriptstyle \Hil$}}}

\newcommand{\laplace}{\Delta^{\!\!\nabla}}

\newcommand{\taut}{\mathscr T}

\numberwithin{equation}{section}  

\newtheoremstyle{style1}
  {10pt}
  {3pt}
  {\it}
  {}
  {\bf\scshape}
  {.}
  {3pt}
  {}
\newtheoremstyle{style2}
  {10pt}
  {3pt}
  {\it}
  {}
  {\bf}
  {:}
  {3pt}
  {}
\theoremstyle{style1}

\newtheorem*{definition*}{Definition}
\newtheorem{theorem}{Theorem}

\newtheorem{lemma}{Lemma}
\theoremstyle{style2}

\usepackage[
  breaklinks,
  colorlinks=true,
  linkcolor=black,
  anchorcolor=black,
  citecolor=black,
  filecolor=black,
  menucolor=black,
  urlcolor=blue,
  bookmarks,
  bookmarksnumbered,
  hyperfootnotes=false,
  pdfpagelabels,
  pdfstartview=FitH,
  pdfpagemode=UseOutlines,
  pdfnewwindow=true,
  pagebackref=false,
  bookmarksopen=false,
  bookmarks=true
]{hyperref}
\usepackage[figure]{hypcap}

\begin{document}

\title[Random Zero Currents]{Random Zero Currents of Sections of Hermitian Line Bundles over Compact Riemannian Manifolds}
\author{Felix Kn{\"o}ppel}

\address{Felix Kn\"oppel\\
  Institut f\"ur Mathematik\\ Technische Universit\"at Berlin\\
  Stra{\ss}e des 17.\ Juni 136\\
  10623 Berlin\\ Germany}

\email{knoeppel@math.tu-berlin.de}

\date{\today}
\thanks{The author was supported by DFG SFB/TRR 109 ``Discretization in Geometry and Dynamics'' and the Einstein Foundation Berlin.}

\maketitle

\begin{abstract}
This paper is concerned with zero currents of random section of a Hermitian line bundle \(E\) over a compact oriented Riemannian manifold. Given a metric connection, heat flow yields a natural 1-parameter family of probability measures on the space of smooth sections \(\Gamma E\). It is shown that the corresponding family of random zero currents connects the curvature of the bundle to the ground state zero current.
\end{abstract}

\section{Introduction}

Throughout, let \(E\) denote a Hermitian line bundle with metric connection \(\nabla\) over a compact oriented Riemannian manifold \(M\) of dimension \(m\geq 2\). Generically, a section \(\psi\in\Gamma E\) has simple zeros which form a submanifold \(\Lambda\subset M\) of codimension \(2\). The zero current \(\zcurr_\psi\) of \(\psi\) is then given by integration over \(\Lambda\), i.e. \(\zcurr_\psi = \delta_\Lambda\), where the orientation is such that \(\psi\) restricted to a small positively oriented circle around \(\Lambda\) has positive degree.

In case that \(E\) is a positive holomorphic Hermitian line bundle over a compact complex manifold \(M\) the curvature \(R^\nabla = d^{\nabla} d^{\nabla}\) is obtained as a weak limit of zero currents of randomly chosen holomorphic sections of the tensor powers \(E^N\) \cite{zelditch1998assymptotics,shiffman1999distribution}: For \(\eta\in \Omega^{m-2}M\),
\[
	\lim_{N\to \infty}\tfrac{1}{N}\int_{H^0(M,E^N)} \langle\zcurr |\eta\rangle\, d\mu_N = \tfrac{1}{2\pi}\int_M JR^\nabla\!\!\wedge\eta\,,
\]
where \(\langle.|.\rangle\) denotes the dual pairing and \(\mu_N\) denotes the standard Gaussian measure on the finite-dimensional space of holomorphic sections \(H^0(M;E^N)\).

As observed in experiments \cite{knoeppel2013god} the zero currents of low energy states of the Laplacian \(\laplace = -\star d^\nabla\!\!\star d^\nabla\) concentrate according to curvature as well. In \cite{weissmann2014smoke,eberhardt2017smoke} this observation was used to decompose smoke into smoke rings. Though this decomposition works quite well in practice, it is an open question how exactly the zero current of a low energy state is related to curvature. Our main result---a natural curve of random zero currents which connects curvature on the one hand to the ground state zero current on the other hand---is a small first step towards an answer.

In \cite{nicolaescu2016,nicolaescu2017} Nicolaescu et. al. construct Gaussian measures on the space of smooth sections to prove a stochastic Gauss--Bonnet--Chern formula. The present paper uses somewhat similar but simpler techniques. In our situation, given a Riemannian metric, heat flow provides us a priori with a natural family of Gaussian measures on the Hilbert spaces of square-integrable sections, all of which are concentrated on \(\Gamma E\). The heat kernel embedding then translates the problem to an integral geometric problem in infinite-dimensional complex projective space, which adds a nice geometric flavor.

\section{Main Results}

The heat semigroup of the connection Laplacian \(\laplace\) assigns to each \(t>0\) a self-adjoint trace class operator \(S_t\) on the Hilbert space of square-integrable sections \(\Hil:=\Gamma_{L^2}E\) \cite{berline2003heat},
\[
	S_t\colon \Hil \to \Gamma E \subset \Hil\,.
\]
This gives rise to a family of Gaussian probability measures \(\mu_t\) concentrated on the space of smooth sections \(\Gamma E\). The main goal of this paper is then to prove the following theorem.

\begin{theorem}\label{mainthm:limits-of-zero-distributions}
	Let \(R^\nabla\) denote the curvature of \(\nabla\) and \(\zcurr_0 \in (\Omega^{m-2}M)^\ast\) the ground state zero current of the Laplacian  \(\laplace\). Then, for each \(\eta\in \Omega^{m-2}M\),
	\[
		\lim_{t \to 0^+}\int_{\Gamma E}\langle\zcurr|\eta\rangle \,d\mu_t= \tfrac{1}{2\pi}\int_M JR^\nabla\!\!\wedge\eta \,, \quad \lim_{t\to \infty}\int_{\Gamma E} \langle\zcurr|\eta\rangle \,d\mu_t = \langle\zcurr_0|\eta\rangle \,.
	\]
\end{theorem}

The proof of Theorem \ref{mainthm:limits-of-zero-distributions} is given in Section \ref{sec:limits}. It is based on the heat kernel embeddings \(f_t\) (Section \ref{sec:Kodaira-corr}) of \(M\) into the infinite-dimensional complex projective space \(\CP^\infty=\P(\Hil)\) and follows from two observations both of which, we think, are of interest on their own:

First, an infinite-dimensional version of the following integral geometric theorem.

\begin{theorem}\label{mainthm:manifold-hyperplane-intersection-cpn}
	Let \(f\colon M\to \CP^n\) be an immersed oriented compact manifold with boundary. Let \(\omega_K\) denote the K{\"a}hler form of \(\CP^n\) and \(\eta\in\Omega^{m-2} M\). Then
	\[
		\pi\int_{\CP^n} \langle\icurr_f|\eta\rangle \, d\vol_{\CP^n} = \vol(\CP^n)\int_M f^\ast\omega_K\!\wedge \eta \,.
	\]
\end{theorem}

Here \(\sigma_{f,[\psi^\perp]}\) denotes the intersection current of \(f(M)\) and the projective hyperplane \([\psi^\perp]\) as defined  in Section \ref{sec:Kodaira-corr}. The theorem is proven in Section \ref{sec:Imm-into-CPn} and can be considered as modification of results of Tasaki \cite{tasaki1995intgeom} and Kang et al \cite{tasaki2001intgeom}. Its infinite-dimensional counterpart will then be shown in Section \ref{sec:Imm-into-CPinfty}.

Second, as shown in Section \ref{sec:limits}, the pullback of the tautological line bundle \(\taut\) of \(\CP^\infty\) under \(f_t\) becomes---as Hermitian line bundle with connection---isomorphic to \(E\) when \(t\to 0_+\). More precisely, we give an explicit isomorphism \(F_t\colon E\to f_t^\ast\taut\) of Hermtian line bundles and show the following theorem using heat kernel parametrices.

\begin{theorem}\label{mainthm:t-to-zero-limit}
		If \(\nabla^t\) denotes the induced connection on \(\mathrm{Hom}(E,f_t^\ast\taut)\) then, for all \(\ell\geq 0\),
		\[
			\Vert F_t^{-1}\nabla^t F_t\Vert_{C^\ell} \to 0 \enspace \textit{for}\enspace t\to 0_+ \,.
		\]
\end{theorem}

\section{Heat Kernel Embedding and Intersection Currents}\label{sec:Kodaira-corr}

Consider the trivial bundle \(\underline{\Gamma E}_M = M\times \Gamma E\). The evaluation map \(\mathrm{ev}\colon \underline{\Gamma E}_M \to E\), \((p,\psi)\mapsto \psi_p\), defines a natural surjective bundle map. Its dual then yields an identification of \(E^\ast\) with a line subbundle of the trivial bundle \(\underline{\Gamma E}_M^\ast\), which is the same thing as an injection \(f_0\colon M\to\P((\Gamma E)^\ast)\). This map is known as \emph{Kodaira correspondence} \cite{ferus2001}. 

For us it is convenient to identify \(E^\ast\) with \(E\) and work with the \emph{Dirac delta} \(\delta\colon E\to (\Gamma E)^\ast\) which assigns to a vector \(\xi \in E_p\) the Dirac distribution given by \(\langle\delta_\xi | \psi\rangle := \langle \xi,\psi_{p}\rangle_\subE\), where \(\langle.,.\rangle_\subE\) denotes the Hermitian metric of \(E\). Clearly, \(f_0(p) = \delta_{E_p}\). Moreover, one easily checks that \(f_0\) relates the zeros of a section \(\psi \in \Gamma E\) to intersections of \(f(M)\) and the projective hyperplane
\[
	[\psi^\perp] = \{[\phi] \in \P((\Gamma E)^\ast) \mid \langle\phi| \psi\rangle = 0\} \,.
\]
Now, smoothing the Dirac delta by the heat flow \(S_t\) defines a map \(\delta^t := S_t\delta\), which is known to take values in \(\Gamma E\subset \Hil\) \cite{berline2003heat}. Here \(\Hil^\ast\) and \(\Hil\) are identified via the Hermitian product \(\langle.,.\rangle_\subH\) of \(\Hil\). The map \(\delta^t\) thus descends to a map \(f_t\) mapping \(M\) into the infinite-dimensional complex projective space \(\CP^\infty:= \P(\Hil)\),
\[
	\begin{xy}
 	\xymatrix{
		E^\times\ar[r]^{\delta^t = S_t\delta} \ar[d]    &   \Hil\setminus \{0\} \ar[d]  \\
  		M \ar[r]_{f_t}  & \CP^\infty
  	}
	\end{xy}
\]
where \(E^\times\) denotes \(E\) with its zero section removed. As in \cite{gallot1994} one proves the following theorem which justifies to call \(f_t\) \emph{heat kernel embedding}.

\begin{theorem}
	For each \(t>0\) the map \(f_t\colon M \to \CP^\infty\) is a smooth embedding.
\end{theorem}

For \(\psi\in\Hil\) and \(\xi\in E\) we have that \(\langle S_t\delta_\xi, \psi\rangle_\subH = \langle \delta_\xi | S_t\psi\rangle\). Thus \(f_t\) relates the zeros set of the section \(S_t\psi\) to the intersection points of \(f_t(M)\) and the projective hyperplane \([\psi^\perp]\).

More generally, if \(f\colon M \to \CP^\infty\) is a smooth immersion and \(H \subset\CP^\infty\) is a projective hyperplane transverse to \(f\), then the preimage \(\Lambda = f^{-1}H\subset M\) is a compact submanifold of codimension \(2\). By transversality, for each \(p\in \Lambda\) there is a complex line \(L \subset \T_{f(p)}\CP^\infty\) such that \(\T_{f(p)}\CP^\infty = L\oplus \T_{f(p)} H\) and, if \(\pi_L\) denotes the projection onto \(L\), then \(\pi_L\circ d_pf\) restricts to an isomorphism between the normal space \(\N_p\Lambda\) and \(L\). Therefore the canonical orientation of \(L\) carries over to \(N\Lambda\) and consequently determines a canonical orientation for \(\Lambda\). The {\it intersection current} \(\icurr_{f,H}\) is then defined by integration over \(\Lambda\) with its canonical orientation.

Sometimes it is convenient to work with the Fubini--Study metric, which is the unique Riemannian metric \(\langle.,.\rangle_\RR\) on \(\CP^\infty\) that turns the canonical projection \(\pi\) from the infinite unit sphere \(\mathbb S^\infty\) onto \(\mathbb S^\infty/\mathbb S^1 \cong \CP^\infty\) into a Riemannian submersion (cf. \cite{gallot2004riemannian}). This turns \(\CP^\infty\) into a K{\"a}hler manifold with Hermitian metric \(\langle.,.\rangle = \langle.,.\rangle_\RR + i\langle J.,.\rangle_\RR\). In this situation the orientation of an intersection current \(\icurr_{f,H}\) is such that
\[
	{\det}_\RR (\pr_{\N H}\circ d f_t|_{\N\Lambda}) > 0\,,
\]
where \(\pr_{\N H}\) denotes the orthogonal projection to \(\N H\) and \(\det_\RR\) the real determinant. 

\begin{lemma}\label{lma:zero-intersection-relation}
	Let \(f_t\) be the heat kernel embedding for time \(t>0\) and \([\psi^\perp]\subset \CP^\infty\) a projective hyperplane transverse to \(f_t\), then 
	\[
		\icurr_{f_t,[\psi^\perp]} = -\zcurr_{\,S_t\psi} \,.
	\]
\end{lemma}

\begin{proof}
	As pointed out above, \(f_t^{-1}[\psi^\perp]\) coincides with the zero set \(Z\) of \(S_t\psi\). 	Given a local frame \(\xi\in \Gamma E\) at \(p\in Z\), there is \(\CC\)-valued function \(\varphi\) such that \(S_t\psi = \varphi\,\xi\) and we are done if we can show that, for any orientation on \(Z\),
	\[
		\mathrm{sign}\,{\det}_\RR (\pr_{\N[\psi^\perp]}\circ d f_t|_{\N_p\Lambda}) = - \mathrm{sign}\,{\det}_\RR(d\varphi|_{\N_p\Lambda})\,.
	\]
	The local frame \(\xi\) can be chosen such that \(\hat f_t = \delta^t_\xi\) lifts \(f_t\) to \(\mathbb S^\infty\subset \Hil\) with \(d_p \hat f_t\) taking values in the horizontal space \(\hat f_t(p)^\perp\):
	\[
		\hat f_t \colon M \to \mathbb S^\infty ,\quad f_t = \pi\circ \hat f_t\,,\quad \mathrm{img}\,d_p \hat f_t \subset \hat f_t(p)^\perp \,,
	\]
	Since \(\langle \psi,\hat f_t(p) \rangle_\subH=0\), \(\psi\in \T_{\hat f_t(p)}\mathbb S^\infty\) is horizontal and \(d_{\hat f_t(p)}\pi(\psi) \in \N_{f_t(p)}[\psi^\perp]\). Using that \(\pi\) is a Riemannian submersion, we get
	\begin{gather*}
		\pr_{\N_{f_t(p)}[\psi^\perp]}\circ d_p f_t = \langle d_{\hat f_t(p)}\pi(\psi), d_p f_t \rangle = \langle \psi, d_p \hat f_t \rangle_\subH \\
		= d_p\langle\psi,\hat f_t\rangle_\subH = d_p\langle \psi,\delta^t_\xi\rangle_\subH = d_p\langle S_t\psi ,\xi \rangle_\subE = d_p(|\xi|^2\bar\varphi)\,.
	\end{gather*}
	Thus \(\mathrm{sign}\,{\det}_\RR (\pr_{\N[\psi^\perp]}\circ d f_t|_{\N_p\Lambda}) = \mathrm{sign}\,{\det}_\RR (d(|\xi|^2\bar\varphi)|_{\N_p\Lambda}) = -\mathrm{sign}\,{\det}_\RR (d\varphi|_{\N_p\Lambda})\).
\end{proof}

\section{Immersions into \(\CP^n\)}\label{sec:Imm-into-CPn}

In finite dimensions the (normalized) Riemannian volume form of \(\CP^n\) provides a canonical rotationally invariant probability measure on the space of projective hyperplanes. The following theorem tells us that the intersection current of an immersed compact oriented manifold \(f\colon M\to \CP^n\) and a random hyperplane is (up to a factor of \(\pi\)) given by the pullback of the K\"ahler form \(\omega_K\) of \(\CP^n\),
\[
	\omega_K = \imaginary\langle.,.\rangle \,.
\]
It can be regarded as a modification of a result of Kang and Tasaki \cite{tasaki2001intgeom} on real surfaces in complex projective space.

\begin{theorem}\label{thm:manifold-hyperplane-intersection-cpn}
	Let \(f\colon M\to \CP^n\) be a compact immersed oriented manifold with boundary and \(\eta\in \Omega^{m-2} M\), \(m=\dim M\). Then
	\[
		\pi\int_{\CP^n}\langle\icurr_f |\eta\rangle\, d\vol_{\CP^n} = \vol(\CP^n)\int_M f^\ast\omega_K\!\wedge \eta \,.
	\]
\end{theorem}

\begin{proof}
	Let \(\taut\subset \CP^n\times \CC^{n+1}\) denote the tautological bundle and \(\taut^\perp\) its orthogonal bundle,
	\[
		\CP^n\times \CC^{n+1} = \taut \oplus \taut^\perp \,.
	\]
	The space of projective hyperplanes intersecting \(f(M)\) can be parametrized by the projectivized pullback bundle \([f^\ast\taut^\perp] \subset M\times \CP^n\) via projection onto the second component
	\[
		\pi_2 \colon [f^\ast\taut^\perp] \to \CP^n,\quad (p,[\psi]) \mapsto [\psi] \,.
	\]
	Let \(\pi_1\colon [f^\ast\taut^\perp] \to M\) denote the projection to the first component. Since the differentials \(d\pi_i\) have constant rank everywhere, the corresponding kernel bundles are smooth subbundles and provide a splitting of \(T [f^\ast\taut^\perp]\):
	\[
		T[f^\ast\taut^\perp] = V\oplus H,\quad V = \ker d\pi_1,\quad H = \ker d\pi_2 \,.
	\]
	In particular, \(H \cong \pi_1^\ast \T M\) and \(\T [f^\ast\taut^\perp]_p \cong V|_{[f^\ast\taut^\perp]_p}\).
	Since \([f^\ast\taut^\perp]\) is a \(\CP^{n-1}\)-fiber bundle each fiber caries a canonical orientation which, by the splitting, extends to a bundle orientation \(\alpha\in\Omega^{2n-2}[f^\ast\taut^\perp]\). Furthermore, if \(d\vol_M\) denotes a volume form on \(M\), then the total space \([f^\ast\taut^\perp]\) is oriented by the local product orientation \cite{greub1972connections},
	\[
		d\vol_{[f^\ast\taut^\perp]} = \pi_1^\ast d\vol_M\wedge \alpha \,.
	\]
	Since the form \(\pi_2^\ast d\vol_{\CP^n}\) is of even degree, the coarea formula (Theorem \ref{thm:generalized-degree-formula}) yields
	\[
		\int_{\CP^n} \langle\pcurr_{\pi_2}| \pi_1^\ast\eta\rangle\, d\vol_{\CP^n} = \int_{[f^\ast\taut^\perp]}\pi_1^\ast\eta\wedge \pi_2^\ast d\vol_{\CP^n} = \int_M\fint_{[f^\ast\taut^\perp]}\pi_1^\ast\eta\wedge \pi_2^\ast d\vol_{\CP^n} \,.
	\]
	Here \(\pcurr_{\pi_2}\) is the preimage current defined in Appendix \ref{app:degthm} and \(\fint_{[f^\ast\taut^\perp]}\) denotes the integral over the fiber \cite{greub1972connections}.
	
	From the definition of \(\pcurr_{\pi_2}\), \(\icurr_f\) it's not hard to check that \(\langle\pcurr_{\pi_2}| \pi_1^\ast\eta\rangle = \langle\icurr_{f}| \eta\rangle\). By Proposition IX in Section 7.13 of \cite{greub1972connections} we get
	\[
		\fint_{[f^\ast\taut^\perp]}\pi_1^\ast\eta\wedge \pi_2^\ast d\vol_{\CP^n} = \eta\wedge \fint_{[f^\ast\taut^\perp]}\pi_2^\ast d\vol_{\CP^n} \,.
	\]
	Thus we are done if we can show that
	\[
		\fint_{[f^\ast\taut^\perp]} \pi_2^\ast d\vol_{\CP^n} = \tfrac{\vol(\CP^n)}{\pi}f^\ast\omega_K \,.
	\]
	To perform the integration over the fiber we compute the \emph{retrenchment} \cite{greub1972connections} of \(\pi_2^\ast d\vol_{\CP^n}\), i.e. a section \(\Omega \in \Gamma\,\mathrm{Hom}(\Lambda^{2n-2} V, \pi_1^\ast\Lambda^2\T M^\ast)\) such that
	\[
		\Omega(Z_1,\ldots,Z_{2n-2})(d\pi_1(X),d\pi_1(Y)) = \pi_2^\ast d\vol_{\CP^n}(X,Y, Z_1,\ldots,Z_{2n-2})\,,
	\]
	where \(X,Y\in\Gamma T[f^\ast\taut^\perp]\) and \(Z_i \in \Gamma V\).
	
	Therefore, let \((p_0,[\psi_0])\in [f^\ast\taut^\perp]\). Assume that \(|\psi_0| =1\) and let \(\phi_0\in \mathbb S^{2n+1}\) such that \(f(p_0) = [\phi_0]\). Then there is a neighborhood \(U\ni p_0\) and a local lift \(\hat f \colon U\to \phi_0+\phi_0^\perp\subset \CC^{n+1}\) of \(f\). Consider the parametrization
	\[
		\varphi\colon U\times \phi_0^\perp \to [f^\ast\taut^\perp], \quad (p,\psi) \mapsto (p,[\widehat\varphi(p,\psi)]) \,,
	\]
	where
	\[
		\widehat\varphi (p,\psi) := |\hat f(p)|^2\psi - \langle\widehat f(p), \psi\rangle \widehat f(p) \,.
	\]
	Let \(\dot p\in \T_{p_0} M\), \(\dot\psi \in \T_{\psi_0} \phi_0^\perp\cong \phi_0^\perp\). Then \(d\widehat f(\dot p) \in \phi_0^\perp\) and
	\begin{align*}
		d\widehat\varphi (\dot p, \dot\psi) & = 2\real(\underbrace{\langle d\widehat f(\dot p),\phi_0\rangle}_{=0})\psi_0 + \dot\psi - \bigl(\langle d\widehat f(\dot p), \psi_0\rangle + \underbrace{\langle\phi_0, \dot\psi\rangle}_{=0}\bigr)\phi_0 - \underbrace{\langle\phi_0,\psi_0\rangle}_{=0} d\widehat f(\dot p)\\
& = \dot\psi - \langle d\widehat f(\dot p), \psi_0\rangle \phi_0 \,. 
	\end{align*}
	Thus \(d_{(p_0,\psi_0)}\widehat\varphi\colon \mathrm T_{p_0} M \oplus \phi_0^\perp \to \CC\phi_0 \oplus \phi_0^\perp\) has the following form
	\[
		d_{(p_0,\psi_0)}\widehat\varphi= \left(\begin{matrix}-\langle d_{p_0}\widehat f, \psi_0\rangle & 0 \\ 0 & \mathrm{id}_{\phi_0^\perp}  \end{matrix}\right) \,.
	\]
	Now, let \(X,Y\in \T_{p_0} M\) and let \({\dot\psi}_1,\ldots,{\dot\psi}_{2n-2} \in \phi_0^\perp \cap \psi_0^\perp\) be a positively oriented orthonormal basis. Set \(\widehat X = d\varphi(X)\), \(\widehat Y = d\varphi(Y)\) and \(Z_i = d\varphi(\dot\psi_i)\) for \(i=1,\ldots,m\). Then \(d\pi_1(\widehat X) = X\), \(d\pi_1(\widehat Y) = Y\) and \(Z_i \in V_{(p_0,[\psi_0])}\).
	Note, by the Fubini--Study construction the canonical projections \(\pi_{\CP^n} \colon \CC^{n+1}\setminus \{0\} \to \CP^n\) and \(\pi_{\CP^{n-1}}\colon \phi_0^\perp\setminus \{0\} \to [\phi_0^\perp]\) provide complex linear isometries
	\[
		d_{\phi_0} \pi_{\CP^n}\colon \phi_0^\perp \to T_{f(p)}\CP^n,\quad d_{\psi_0} \pi_{\CP^{n-1}}\colon \phi_0^\perp \cap \psi_0^\perp \to T_{[\psi_0]}[\phi_0^\perp] \,.
	\]
	Thus \(\alpha(Y_{m+1},\ldots,Y_{m+2n-2})=1\) and we get
	\begin{align*}
		\Omega(Z_1,\ldots,Z_m) & (X,Y)=  \pi_2^\ast d\vol_{\CP^n} (d\varphi(X),d\varphi(Y),Z_{m+1},\ldots,Z_{m+2n-2}) \\
			& = (\pi_{\CP^n} \circ\widehat\varphi)^\ast d\vol_{\CP^n} (X,Y,{\dot\psi}_1,\ldots,{\dot\psi}_{2n-2}) \\
			& = \widehat\varphi^\ast d\vol_{\mathbb S^{2n+1}}(J\psi_0,X,Y,{\dot\psi}_1,\ldots,{\dot\psi}_{2n-2}) \\
			& = \widehat\varphi^\ast d\vol_{\CC^{n+1}}(\psi_0,J\psi_0,X,Y,{\dot\psi}_1,\ldots,{\dot\psi}_{2n-2}) \\
			& = \langle d_{p_0}\widehat f(X),\psi_0\rangle_\RR\langle d_{p_0}\widehat f(Y),J\psi_0\rangle_\RR - \langle d_{p_0}\widehat f(Y),\psi_0\rangle_\RR\langle d_{p_0}\widehat f(X),J\psi_0\rangle_\RR \\
			& = \langle d_{p_0}\widehat f,\psi_0\rangle_\RR\wedge\langle d_{p_0}\widehat f,J\psi_0\rangle_\RR (X,Y)\, \alpha(Z_{m+1},\ldots,Z_{m+2n-2}) \,.
	\end{align*}
	Using that \(d_{\phi_0} \pi_{\CP^n} \colon \phi_0^\perp \to T_{f(p_0)}\CP^n\) sends \(\psi_0\) to a unit normal vector \(\nu \in \N_{f(p_0)}[\psi^\perp]\), we get
	\[
		 \langle d_{p_0}\widehat f,\psi_0\rangle_\RR\wedge\langle d_{p_0}\widehat f,J\psi_0\rangle_\RR = \langle d_{p_0}f,\nu\rangle_\RR\wedge\langle d_{p_0} f,J\nu\rangle_\RR = \langle Jd_{p_0}f\wedge \langle d_{p_0} f,\nu\rangle_\RR\nu\rangle_\RR \,.
	\]
	Moreover, \(\langle Jd_{p_0}f\wedge\langle d_{p_0} f,\nu\rangle_\RR \nu\rangle_\RR/|\nu|^2\) is independent of the choice of \(\nu\) and thus only depends on \((p_0,[\psi_0])\in [f^\ast\taut^\perp]\). This defines a smooth bundle map \(A\colon [f^\ast \taut^\perp] \to \Lambda^2 \T^\ast M\). With the notation in Section 7.13 of \cite{greub1972connections},
	\[
		\Omega = A\alpha \in \Gamma\,\mathrm{Hom}(\Lambda^{2n-2}V, \pi_1^\ast\Lambda^2\T^\ast M) \,.
	\]
	Note that the unit vector bundle \(\widehat\pi\colon \U([f^\ast\taut^\perp]) \to [f^\ast\taut^\perp]\), consisting of all the representatives of length one, is an orientable \(\mathbb S^1\)-bundle. Let \(d\theta\in\Omega^1\U(f^\ast\taut^\perp)\) denote the extended Riemannian fiber volume. Then \(\widehat\pi^\ast\Omega\wedge d\theta\) orients \(\U(f^\ast \taut^\perp)\to M\) and
	\[
		\pi\fint_{[f^\ast\taut^\perp]} \Omega = \fint_{\U[f^\ast \taut^\perp]} \widehat\pi^\ast(A\alpha)\wedge d\theta = \bigl\langle Jd\widehat f\wedge \fint_{\U[f^\ast \taut^\perp]} \mathrm{pr}_{\RR(.)}\circ d\widehat f\, \bigr\rangle_\RR \,.
	\]
	Here \(\mathrm{pr}_{\RR \psi}\) denotes the orthogonal projection to real span of \(\psi\in f^\ast\taut^\perp\).
	
	Given \(\psi\in \U[f^\ast\taut^\perp]_{p_0}\) then there exist rotations \(\rho_1,\ldots,\rho_{2n}\) such that  \(\rho_1\psi,\ldots,\rho_{2n}\psi\) form an orthonormal basis
	\begin{gather*}
		\int_{\mathbb S^{2n-1}} \pr_{\RR(.)}\circ d_{p_0}\widehat{f} = \tfrac{1}{2n}\sum_i \int_{\mathbb S^{2n-1}} \pr_{\RR \rho_i(.)}\circ d_{p_0} \widehat f	= \tfrac{1}{2n}\int_{\mathbb S^{2n-1}} d_{p_0}\widehat f = \tfrac{\vol(\mathbb S^{2n-1})}{2n}d_{p_0}\widehat f \,.
	\end{gather*}
	With \(\langle J d_{p_0}\widehat f\wedge d_{p_0}\widehat f\rangle_\RR = \langle Jd_{p_0}f\wedge d_{p_0}f\rangle_\RR = 2f^\ast\omega_k|_{p_0}\) and \(\vol(\CP^n) = \frac{\pi^n}{n!}=\frac{\pi}{n}\vol(\CP^{n-1})\) follows the claim.
\end{proof}

\section{Immersions into \(\CP^\infty\)}\label{sec:Imm-into-CPinfty}

Let \(S\colon \Hil \to \Hil\) be a selfadjoint positive-definite trace-class operator. To each such operator we can assign a centered Gaussian probability measure \(\mu\) on \(\Hil\) uniquely characterized by its Fourier transform \cite{parthasarathy1972probability}
\[
	\widehat\mu(\psi) = \int_\Hil e^{i\langle\phi,.\rangle_\RR}d\mu = e^{-\frac{1}{2}|S\psi|^2} \,.
\]
Suppose \(S\) has finite-dimensional eigenspaces and possesses a complete orthonormal system \(\{\psi_i\in\Hil\}_{i\in\NN}\) of eigenvectors:
\[
	S\psi_i = \lambda_i\psi_i, \quad \lambda_1\geq \lambda_2\geq \lambda_3\geq \ldots > 0 \,.
\]
Then \(\Hil_k = \mathrm{span}\{\psi_i\mid i\leq k\}\) forms a nested sequence of finite-dimensional subspaces and with \(\CP^k := \P(\Hil_k)\) we obtain a commuting diagram
\[
	\begin{xy}
 	\xymatrix{
		\Hil_1\setminus\{0\} \ar[d] \ar[r] & \Hil_2\setminus\{0\} \ar[d] \ar[r]  & \Hil_3\setminus\{0\} \ar[d]  \ar[r] & \cdots   \Hil\setminus\{0\} \ar[d]  \\
  		\CP^0 \ar[r] & \CP^1 \ar[r] & \CP^2 \ar[r] & \cdots \CP^\infty \\
  	}
	\end{xy}
\]
Let \(\widetilde S_k\colon \Hil_k\to \Hil_k\) denote the restriction of \(S\) to \(\Hil_k\) and let \(S_k := \widetilde S_k \circ \pr_k\), where \(\pr_k\colon \Hil \to \Hil_k\) denotes the orthogonal projection onto \(\Hil_k\). Then, for each \(k\), there is a Gaussian measures \(\widetilde\mu_k\) on \(\Hil_k\) corresponding to \(\widetilde S_k\) and a Gaussian measures \(\mu_k\) on \(\Hil\) corresponding to \(S_k\). For \(k\to\infty\), we have strong convergence \(S_k \to S\). From Lemma 5.1 and Theorem 4.5 of Chapter VI in \cite{parthasarathy1972probability} follows weak convergence of the measures:
\[
	\mu_k \to \mu \textit{ for } k\to \infty \,.
\]

Recall : Given two topological spaces \(X\) and \(Y\) and a continuous map \(f\colon X \to Y\). Then, for any bounded (or bounded
from below) Borel measure \(\nu\) on X, the pushforward
\[
	f_\ast\nu = \nu \circ f^{-1}
\]
defines a Borel measure on \(Y\) and we have the following transformation formula for pushforward measures (cf. \cite{bogachev2006measure}).

\begin{theorem}\label{thm:transformation-formula-for-pushforward}
Let \(\nu\) be non-negative. A Borel function \(g\) on \(Y\) is integrable with respect to \(f_\ast\nu\) precisely
when the function \(g\circ f\) is integrable with respect to \(\nu\). In addition,
\[
	\int_Y g\, d(f_\ast\nu) = \int_X g\circ f\, d\nu \,.
\]
\end{theorem}

The following lemma will be the key for the transition of Theorem \ref{thm:manifold-hyperplane-intersection-cpn} to infinite dimensions.
\begin{lemma}\label{lma:transition-to-finite-case}
	Let \(\iota_k\colon \Hil_k \to \Hil\) denote the inclusion. Then \(\mu_k = {\iota_k}_\ast \widetilde\mu_k\).
\end{lemma}

\begin{proof}
	It is enough to check that the Fourier transformations coincide. By Theorem \ref{thm:transformation-formula-for-pushforward},
	\begin{gather*}
		\widehat{{\iota_k}_\ast\widetilde{\mu_k}}(\psi) = \int_{\Hil} e^{i\langle\psi,.\rangle_\RR}d({\iota_k}_\ast\widetilde{\mu})  = \int_{\Hil_k} e^{i\langle\psi,\iota_k\rangle_\RR}d\widetilde{\mu} = \int_{\Hil_k} e^{i\langle\pr_k\psi,.\rangle_\RR}d\widetilde{\mu} \\= e^{-\frac{1}{2}|\widetilde{S}_k\pr_k\psi|^2} = e^{-\frac{1}{2}|S_k\psi|^2} = \widehat{\mu_k} (\psi)
	\end{gather*}
\end{proof}

Let \([\mu]\) denote the pushforward measure of \(\mu\) with respect to the canonical projection \(\pi\colon \Hil\setminus\{0\}\to \CP^\infty\) and let \(\omega_K^g\) denotes the K\"ahler form of the Fubini-Study metric corresponding to the pullback product \(g:=S^\ast\langle.,.\rangle\).

\begin{theorem}\label{thm:manifold-hyperplane-intersection-cpinfty}
	Let \(f\colon M \to \CP^\infty\) be a compact immersed oriented manifold with boundary and \(\eta\in \Omega^{m-2} M\), \(m=\dim M\). Then
	\[
		\pi\int_{\CP^\infty} \langle\icurr_{f}| \eta\rangle\, d[\mu] = \int_M f^\ast\omega^g_K\wedge \eta \,.
	\]	
\end{theorem}

\begin{proof}
	For a \([\mu_k]\)-measurable function \(\varphi\) on \(\CP^\infty\), we get by Lemma \ref{lma:transition-to-finite-case} and Theorem \ref{thm:transformation-formula-for-pushforward}
	\begin{gather*}
		\int_{\CP^\infty}\varphi\, d[\mu_k] = \int_{\Hil\setminus\{0\}} (\varphi\circ\pi)\, d\mu_k = \int_{\Hil\setminus\{0\}} (\varphi\circ\pi)\, d({\iota_k}_\ast\widetilde\mu_k) = \int_{\Hil_k\setminus\{0\}} (\varphi\circ\pi)\, d\widetilde\mu_k \,.
	\end{gather*}
	The operator \(\widetilde S_k\) is invertible and selfadjoint. Moreover,
	\[
		\int_{\Hil_k} e^{i\langle\psi,.\rangle_\RR}d{(\widetilde S^{-1}_k)}_\ast \widetilde\mu_k = \int_{\Hil_k} e^{i\langle\widetilde S_k^{-1}\psi,.\rangle_\RR}d\widetilde\mu_k = e^{-\tfrac{1}{2}|\widetilde S_k (\widetilde S_k^{-1} \psi)|^2} = e^{-\tfrac{1}{2}|\psi |^2} \,.
	\]
	 Hence \({(\widetilde S^{-1}_k)}_\ast \widetilde\mu_k\) is the standard Gaussian probabilty distribution which, in particular, is rotationally invariant. Thus its pushforward with respect to the radial projection
	 \[
	 	\rho \colon \Hil_k\setminus\{0\} \to \mathbb S^{2k-1}, \quad \psi \to \tfrac{1}{|\psi|} \psi \,.
	 \]
	 equals \(d\vol_{\mathbb S^{2k-1}}/\vol(\mathbb S^{2k-1})\). Let \(g =S^\ast\langle.,.\rangle\) and \(\mathbb S^\infty_g\subset \Hil\) the round unit sphere with respect to \(g\). Then \(S\colon \Hil \to \Hil\) restricts to an isometry \(\mathbb S^\infty_g \to \mathbb S^\infty\). Moreover, the restriction of \(g\) to \(\Hil_k\) is given by \(g_k = (\widetilde S_k)^\ast\langle.,.\rangle\). In particular, \(\mathbb S^{2k-1}_{g}= \mathbb S^\infty_g\cap \Hil_k\) is the unit sphere with respect to \(g_k\). Now
	\begin{gather*}
		\int_{\Hil_k\setminus\{0\}} (\varphi\circ\pi)\, d\widetilde\mu_k = \int_{\Hil_k\setminus\{0\}} (\varphi\circ\pi\circ \widetilde S_k)\, d(\widetilde S_k^{-1})_\ast \widetilde \mu_k = \tfrac{1}{\vol(\mathbb S^{2k-1})}\int_{\mathbb S^{2k-1}} (\varphi\circ\pi\circ \widetilde S_k)\, d\vol_{\mathbb S^{2k-1}}\\
			= \tfrac{1}{\vol(\mathbb S^{2k-1})}\int_{\mathbb S^{2k-1}_g} (\varphi\circ\pi)\, d\vol_{\mathbb S^{2k-1}_g} = \tfrac{1}{\vol(\CP^{k-1})}\int_{\CP^{k-1}_g} \varphi\, d\vol_{\CP^{k-1}_g} \,.
	\end{gather*}
	We are interested in the function \(\langle\icurr_{f}|\eta\rangle\). Suppose we can construct a sequence of immersions \(\{f_k\colon M\to \CP^k\}_{k>m}\) such that
	\[
		\langle\icurr_f|\eta\rangle|_{\CP^k} = \langle\icurr_{f_k}|\eta\rangle, \quad \lim_{k\to \infty} f_k^\ast\omega_K^g = f^\ast\omega^g_K \,.
	\]
	Then, by Theorem \ref{thm:manifold-hyperplane-intersection-cpn},
	\begin{gather*}
		\pi\int_{\CP^\infty}\langle\icurr_f| \eta\rangle\, d[\mu] =\pi \lim_{k\to \infty}\int_{\CP^\infty}\langle\icurr_f | \eta\rangle\, d[\mu_k] = \tfrac{\pi}{\vol(\CP^{k})}\lim_{k\to \infty}\int_{\CP^{k}_g} \langle\icurr_{f_k} | \eta\rangle\, d\vol_{\CP^{k}_g} \\
		= \lim_{k\to\infty}\int_M f_k^\ast \omega_K^g\wedge\eta = \int_M f^\ast\omega^g_K\wedge \eta \,.
	\end{gather*}
	It is left to show that such sequence exists. Consider the maps \([\pr_k]\colon \CP^\infty\setminus (\CP^k)^\perp \to \CP^k\), \([\psi] \mapsto [\pr_{k+1} \psi]\). Since \((\CP^k)^\perp = \cap_{i=1}^k \psi_i^\perp\) is closed, the sets 
	\[
		\CP^\infty\setminus(\CP^1)^\perp\subset \CP^\infty\setminus(\CP^2)^\perp \subset\ldots\subset \CP^\infty
	\]		
	form an open exhaustion of \(\CP^\infty\). Since \(f(M)\) is compact, \(f(M)\subset \CP^\infty\setminus(\CP^k)^\perp\) for all but finitely many \(k\). Similarly, the preimages under \(d[\pr_k]\) of the complement of the zero sections in \(\T\CP^k\) form an open exhaustion of \(\T\CP^\infty\) with its zero section removed, while the image of the sphere bundle \(\subset\T M\) under \(df\) is compact. Thus there exists \(m\in \NN\) such that \(f(M)\subset \CP^\infty\setminus(\CP^k)^\perp\) for all \(k\geq m\) and the maps \(f_k = [\pr_k]\circ f\) form a sequence of immersions.
	
	Suppose \([\psi]\in \CP^k\) and \([\psi^\perp]\) intersects \(f\) transversally. Since \(\langle \phi, \psi\rangle = \langle \phi, \pr_{k+1}\psi\rangle = \langle \pr_{k+1}\phi, \psi\rangle\), we see that \(f^{-1}([\psi^\perp]) = f_k^{-1}([\psi^\perp])\). If \(\Lambda_i\) denote the (oriented) connected components of \(f^{-1}([\psi^\perp])\) and \(\N\Lambda_i\) their oriented normal bundles, then \(\icurr_{f,[\psi]}\) is given by
	\[
		\icurr_{f,[\psi]} = \sum_i \mathrm{sign}({\det}_\RR(\pr_{\N[\psi^\perp]}\circ d_p f|_{N\Lambda_i}))\, \delta_{\Lambda_i} \,.
	\]
	 Since \(d[\pr_k]\) is complex linear we find that the signs of the determinants do not change. Thus \(\icurr_{f} = \icurr_{f_k}\). That \(\lim_{k\to \infty} f_k^\ast\omega_K^g = f^\ast\omega^g_K\) follows just from \(\lim_{k\to \infty} S_k = S\).
\end{proof}

\section{Limit cases}\label{sec:limits}

Let \(M\) be an oriented Riemannian manifold and \(E\to M\) be a Hermitian line bundle with connection \(\nabla\). To each operator \(S_t \colon \Hil \to \Hil\) of the heat semigroup of the Laplacian \(\laplace\colon \Gamma E \to \Gamma E\) we assign a Gaussian probabilty measure \(\mu_t\) with Fourier transform
\[
	\widehat \mu_t (\psi) = e^{-\tfrac{1}{2}| S_t \psi|^2} \,.
\]
\begin{lemma}\label{lma:S-action-on-mu}
	If \(t = r+s\), \(r,s>0\), then \(\mu_t = {S_r}_\ast \mu_s\). In particular, \(\mu_t(\Hil\setminus\Gamma E) = 0\).
\end{lemma}

\begin{proof}
	The first part follows easily from the semigroup property, the transformation formula for pushforward measures and the fact that the Fourier transformation determines the measure uniquely: For \(t = r+s\) we have \(S_t =S_s\circ S_r\). Thus
	\[
		\int_\Hil e^{i\langle\psi,.\rangle_\RR}d({S_r}_\ast \mu_s) = \int_\Hil e^{i\langle\psi,S_r\rangle_\RR} d\mu_s = \int_\Hil e^{i\langle S_r\psi,.\rangle_\RR} d\mu_s = e^{-\tfrac{1}{2}|S_sS_r\psi|^2} = e^{-\tfrac{1}{2}|S_t\psi|^2} \,.
	\]
	Since \(S_r(\Hil)\subset \Gamma E\), we obtain \(\mu_t(\Hil\setminus\Gamma E) = \mu_s(S_r^{-1}(\Hil\setminus\Gamma E)) = \mu_s(\emptyset) = 0\).
\end{proof}

In particular Lemma \ref{lma:S-action-on-mu} yields \(\mu_t(\Gamma E) = \mu_t(\Hil) = 1\). We are interested in the expectation of the zero current of a \(\mu_t\)-randomly chosen section \(\psi\in \Gamma E\) which, by Lemma \ref{lma:zero-intersection-relation} and Theorem \ref{thm:manifold-hyperplane-intersection-cpinfty}, is related to the K\"ahler form of \(\CP^\infty\).

It is well-known that, in the finite-dimensional case, the curvature \(R^\taut\) of the tautological line bundle over complex projective space is given by \(R^\taut = 2\,\omega_K J\). As one easily checks, this holds verbatim in the infinite-dimensional case. Thus, if \(f_t \colon M \to \CP^\infty\) denotes the heat kernel embedding, the curvature \(R^t\) of the pullback bundle \(f_t^\ast\taut\) is given by
\[
	R^t = 2 f_t^\ast\omega_K J\,.
\]

\begin{theorem}\label{thm:order-vs-curvature}
	Let \(R^t\) denote the curvature of the pullback bundle \(f_t^\ast \taut\) and \(\eta \in \Omega^{m-2} M\), \(m=\dim M\). Then
	\[
		2\pi\int_{\Gamma E} \langle\zcurr|\eta\rangle\, d\mu_t = \int_M JR^t\wedge\eta \,.
	\]
\end{theorem}

\begin{proof}
	 Let \(\omega_K^t\) denote the K\"ahler form of \(\CP^\infty\) induced by the pullback metric \(S_t^\ast\langle.,.\rangle\). In \cite{berline2003heat} it is shown that \(S_t\) satisfies all the assumptions of Theorem \ref{thm:manifold-hyperplane-intersection-cpinfty}. Thus  we have
	\[
		\pi\int_{\CP^\infty} \langle\icurr_{f_t}|\eta\rangle\, d[\mu_t] = \int_M f_t^\ast\omega^t_K \wedge \eta\,.
	\]
	From \(S_tS_t = S_{2t}\) we see that \(f_{t}^\ast\omega^t_K = f_{2t}^\ast\omega_K\). Moreover, by Lemma \ref{lma:zero-intersection-relation}, \(\icurr_{f_t}  = -\zcurr_{S_t}\). With \(r = \tfrac{t}{2}\) and Lemma \ref{lma:S-action-on-mu} we then obtain
	\begin{gather*}
		2\pi\int_{\Gamma E} \langle\zcurr|\eta\rangle\, d\mu_t = 2\pi\int_{\Gamma E} \langle\zcurr_{S_r}|\eta\rangle\, d\mu_{r} = -2\pi\int_{\CP^\infty} \langle\icurr_{f_r}|\eta\rangle\, d[\mu_r] \\
			= - 2\int_M f_r^\ast\omega^r_K\wedge\eta = - 2\int_M f_{2r}^{\,\ast}\omega_K\wedge\eta = \int_M JR^t\wedge \eta \,.
	\end{gather*}
\end{proof}

We are interested in the limit of \(R^t\) for \(t\to 0_+\). To compute this limit we use a heat kernel parametrix.

Let \(E\boxtimes E^\ast := \pi_1^\ast E \otimes \pi_2^\ast E^\ast\), where \(\pi_i\colon M\times M \to M\), \((x_1,x_2)\mapsto x_i\). In \cite{berline2003heat} is shown that there is a heat kernel, i.e. there is a section \(K_t \in \Gamma(E\boxtimes E^\ast)\) continuous over \(\RR_+\times M\times M\) such that
\[
	(e^{-t\laplace} \psi)(x) = \int_{M} K_t(x,.)\psi(.)\, d\vol_M \,.
\]
Moreover it is shown that, given a smooth cut-off function \(\rho\colon \RR_+ \to [0,1]\) such that \(\rho|_{[0,\varepsilon^2/4)}=1\),  \(\rho|_{(\varepsilon^2,\infty)}=0\), then there exist \(\Phi_i\in\Gamma(E\boxtimes E^\ast)\) such that, for every \(N>1\), the kernel \(K^N_t\) defined by
\[
	(4\pi t)^{-1/2}e^{-d(x,y)^2/4t}\rho(d(x,y)^2)j^{-1/2}(x,y)\sum_{i=0}^{N} t^i\Phi_i(x,y)
\]
is asymptotic to \(K_t\), i.e.
\[
	\Vert \partial_t^k (K_t - K_t^N)\Vert_{C^\ell} = O(t^{N-1/2-\ell/2-k}) \,,
\]

The function \(j\) is real and basically given by the determinant of the differential of \(\exp_M\): 
\[
	j(x,y) = |{\det}_\RR(d_\mathbf{x} \exp_y)|,\quad \exp_p(\mathbf x) = x \,.
\]
Moreover, the coefficients \(\Phi_i\) can be determined recursively (cf. \cite{berline2003heat}):
\begin{itemize}
	\item \(\Phi_0(x,y)\) is given by the parallel transport along the geodesic from \(y\) to \(x\)
	\item \(\Phi_i = \Phi_0 U_i\), where the coefficients in the Taylor series of \(U_i\) are polynomials in covariant derivatives of the Riemannian and the bundle curvatures at the point \(y\)
\end{itemize}
The map \(\xi \to \delta^t_\xi\) is \(\CC\)-linear and, with an appropriate scaling, yields an isomorphism \(F_t\in \Gamma\,\mathrm{Hom}(E,f_t^\ast\taut)\) of Hermitian line bundles, which allows us to transport \(f_t^\ast\nabla_\taut\) over to \(E\) where it is given by the usual formula
\[
	\nabla + F_t^{-1}\nabla^t F_t \,.
\]
Here \(\nabla^t\) denotes the induced connection on \(\Gamma\mathrm{Hom}(E,f^\ast\taut)\). When \(F_t\) is restricted to the circle bundle it takes the form 
\[
	\xi \mapsto \frac{\delta^t_\xi}{|\delta^t_\xi|} \,.
\]
\begin{theorem}\label{thm:t-to-zero-limit}
		If \(\nabla^t\) denotes the induced connection on \(\mathrm{Hom}(E,f_t^\ast\taut)\) then, for all \(\ell\geq 0\),
		\[
			\Vert F_t^{-1}\nabla^t F_t\Vert_{C^\ell} \to 0 \textit{ for } t\to 0_+ \,.
		\]
		In particular, \(\Vert R^t - R^\nabla\Vert_{C^\ell} \to 0\) for \(t\to 0_+\).
\end{theorem}

\begin{proof}
	Since \(M\) is compact, it is enough to show the statement locally. In particular we can assume that there exists \(\xi\in\Gamma E\) such that \(|\xi|_E = 1\). Let \(k=0\) and \(N>\ell/2+1\). Then we get \(\Vert K_t - K_t^N\Vert_{C^\ell} = O(t)\) and hence
	\[
		\delta^t_{\xi_p} = S_t \delta_{\xi_p} =\rho_t(.,p)\sum_{i=0}^{N} t^i\Phi_i(.,p)\xi_p + O(t) \,,
	\]
	where
	\[
		 \rho_t(x,y) = (4\pi t)^{-1/2}e^{-d(x,y)^2/4t}\rho(d(x,y)^2)j^{-1/2}(x,y)
	\]
	is real valued. Note that \(\rho_t(p,p) = (4\pi t)^{-1/2}\). Now, since \(F_t\) is unitary, we have
	\[
		F_t^{-1}\nabla^t F_t = \eta_t J
	\]
	for a real-valued form \(\eta_t\in \Omega^1 M\). Then
	\begin{align*}
		\eta_t & = \langle (F_t^{-1}\nabla^t F_t) \xi, J\xi\rangle_{E,\RR} = \langle(\nabla^t F_t) \xi, JF_t\xi\rangle_\RR = \langle d (F_t \xi) - F_t \nabla\xi, JF_t\xi\rangle_\RR \\& = \langle d(F_t \xi), J(F_t \xi)\rangle_\RR - \langle\nabla\xi, J\xi\rangle_{E,\RR} \,.
	\end{align*}
	Let us look at the first term: For \(X\in\T_p M\),
	\[
		\langle X(F_t \xi), J(F_t \xi)\rangle_\RR = \frac{1}{|\delta^t_{\xi_p}|^2}\langle X(\delta^t\circ\xi), J\delta^t_{\xi_p}\rangle_\RR = \frac{1}{|\delta^t_{\xi_p}|^2}X \langle\delta^t\circ\xi, J\delta^t_{\xi_p}\rangle_\RR \,.
	\]
	Now, since \(\rho_{2t}(p,p) = (8\pi t)^{-1/2}\) and \(\Phi_0(p,p)= \mathrm{Id}_{E_p}\),
	\[
		|\delta_{\xi_p}^t|^2 = \langle \delta_{\xi_p}, \delta_{\xi_p}^{2t}\rangle = \langle\xi_p, K_{2t}(p,p)\xi_p\rangle_E = (8\pi t)^{-1/2}(|\xi_p|_E^2 + O(t)) = (8\pi t)^{-1/2}(1 + O(t)) \,.
	\]
	Moreover,
	\begin{gather*}
		X\langle \delta^t\circ\xi, J\delta^t_{\xi_p}\rangle_\RR = X\langle \xi, J\delta^{2t}_{\xi_p}(.)\rangle_{E,\RR} = X\Bigl\langle \xi, J \rho_{2t}(.,p)\sum_{i=0}^{N} (2t)^i\Phi_i(.,p)\xi_p\Bigr\rangle_{E,\RR} + O(t)\\
			= \sum_{i=0}^{N} (2t)^i \bigl(\bigl\langle \nabla_X \xi, J \rho_{2t}(p,p)\Phi_i(p,p)\xi_p\bigr\rangle_{E,\RR} + \bigl\langle \xi_p, J\nabla_X\bigl(\rho_{2t}(,.p)\Phi_i(.,p)\xi_p\bigr)\bigr\rangle_{E,\RR}  + O(t) \,.
	\end{gather*}
	First note that \(X\bigl(\rho_{2t}(p,.)\bigr) = (8\pi t)^{-1/2}O(1)\). Moreover, we have \(\nabla_X\bigl(\Phi_0(.,p)\xi_p\bigr) = 0\)  and \(\langle \xi,J\xi\rangle_{E,\RR}=0\). Thus we obtain
	\[
		\bigl\langle \xi_p, J\nabla_X\bigl(\rho_{2t}(,.p)\Phi_0(.,p)\xi_p\bigr)\bigr\rangle_{E,\RR} = \bigl\langle \xi_p, J\xi_p\bigr\rangle_{E,\RR} X\bigl(\rho_{2t}(.,p)\bigr) = 0 \,.
	\]
	Furthermore,
	\[
		\langle \nabla_X \xi, J \rho_{2t}(p,p)\Phi_i(p,p)\xi_p\rangle_{E,\RR} = (8\pi t)^{-1/2}\langle \nabla_X \xi, J \xi_p\rangle_{E,\RR} \,.
	\]
	Hence
	\[
		X\langle \delta^t\circ\xi, J\delta^t_{\xi_p}\rangle_\RR = (8\pi t)^{-1/2}\Bigl(\langle \nabla_X \xi, J \xi_p\rangle_{E,\RR}+O(t)\bigr)
	\]
	and thus
	\[
		\langle X(F_t \xi), J(F_t \xi)\rangle_\RR = \langle\nabla\xi,J\xi\rangle_{E,\RR} + O(t) \,.
	\]
	Together this yields
	\[
		 \eta_t = \langle X(F_t \xi), J(F_t \xi)\rangle_\RR - \langle\nabla\xi,J\xi\rangle_{E,\RR} = O(t) \,.
	\]
	In particular, \(\Vert R^t - R^\nabla\Vert_{C^\ell} = \Vert d\eta_t\Vert_{C^\ell} \leq \Vert \eta_t\Vert_{C^{\ell+1}}\to 0\) for \(t\to 0_+\).
\end{proof}

Now we finally have everything together to prove Theorem \ref{mainthm:limits-of-zero-distributions}.

\begin{proof}[Proof of Theorem \ref{mainthm:limits-of-zero-distributions}]
	By Theorem \ref{thm:order-vs-curvature} and \ref{thm:t-to-zero-limit} we get, for \(\eta \in \Omega^{m-2}M\),
	\[
		\int_{\Gamma E} \langle \zcurr | \eta\rangle\, d\mu_t = \tfrac{1}{2\pi}\int_M JR^t\wedge \eta \to \tfrac{1}{2\pi}\int_M JR^\nabla\wedge\eta, \textit{ for } t\to 0_+ \,.
	\] 
	For the limit case \(t\to \infty\) consider a complete orthonormal system \(\{\psi_i\}_{i\in \NN}\) of eigenvectors of \(\laplace\),
	\[
		\laplace \psi_i = \lambda_i \psi_i, \quad 0\leq \lambda_0 \leq \lambda_1\leq \ldots
	\]
	then
	\[
		S_t\psi = \sum e^{-t\lambda_i}\langle\psi_i,\psi\rangle\psi_i \,.
	\]
	Let \(A_t \colon \Hil \to \Hil\) be given by \(\psi \mapsto e^{t\lambda_0}\psi\). Then \(A_tS_t = \sum e^{t(\lambda_0-\lambda_i)}\langle\psi_i,.\rangle\psi_i\) converges for \(t\to\infty\) to the orthogonal projection \(\pi_0\) to the eigenspace \(\mathrm{Eig_0} = \ker(\laplace - \lambda_0)\). Let furthermore \(\widetilde \mu_t = (A_t)_\ast\mu_t\). Then, for \(t\to\infty\),
	\[
		\int_\Hil e^{i\langle\psi,.\rangle_\RR}\, d\widetilde \mu_t = \int_\Hil e^{i\langle e^{t\lambda_0}\psi,.\rangle_\RR}\, d\mu_t = e^{-\tfrac{1}{2}|e^{t\lambda_0}S_t\psi|} \to e^{-\tfrac{1}{2}|\pi_0\psi|} =: \widetilde\mu_\infty(\psi) \,,
	\]
	which by the inclusion \(\mathrm{Eig}_0 \to \Hil\) is pushed forward to the Gaussian measure \(\widehat\mu\) of variance \(1\) on \(\mathrm{Eig}_0\). With \(\zcurr_\psi = \zcurr_{e^{t\lambda_0}\psi}\), we get
	\begin{gather*}
		\int_{\Gamma E} \langle\zcurr|\eta\rangle\, d\mu_t = \int_{\Gamma E} \langle\zcurr |\eta\rangle\,d\widetilde\mu_t
			\to \int_{\Gamma E} \langle\zcurr|\eta\rangle\,d\widetilde\mu_\infty = \int_{\mathrm{Eig}_0} \langle\zcurr | \eta\rangle \,d\widehat\mu = \langle \zcurr_0|\eta\rangle \,.
	\end{gather*}
\end{proof}

\appendix

\section{Coarea formula}\label{app:degthm}

Let \(M\) be a compact oriented \(m\)-dimensional Riemannian manifold \(M\) and \(N\) an oriented manifold of dimension \(n\). Let \(f \colon M \to N\) be a smooth map and let \(R\subset M\) denote the regular values of \(f\). Moreover, for \(y\in N\), let \(\Lambda_y: = f^{-1}\{y\}\subset M\).

For each regular value \(y\in N\) the set \(\Lambda_y\) is either empty or an \((m-n)\)-dimensional submanifold of \(M\). Although \(f\) is in general not a fibration, the rank theorem yields that at each regular points \(M\) locally has the structure of a fiber bundle with fiber being an \((m-n)\)-dimensional disk \(D\). In particular, as \(M\) and \(N\) both are oriented, we have away from singular points a canonical product orientation such that for any choice of a horizontal space \(df\) is orientation preserving. 

As \(\Lambda_y\) is covered by such local disk bundles this defines an orientation for \(\Lambda_y\) and thus a preimage current \(\pcurr_{f,y} \in (\Omega^{m-n}M)^\ast\) given by integration over \(\Lambda_y\).

\begin{theorem}\label{thm:generalized-degree-formula}
	Let \(\eta\in \Omega^{m-n}M\) and \(\omega\in \Omega^{n} N\). Then
	\[
		\int_M f^\ast\omega\wedge \eta = \int_N \langle \pcurr_{f}| \eta\rangle\, \omega \,.
	\]
\end{theorem}

\begin{proof}
	Let \(R\subset M\) denote the set of regular points of \(f\). Since \(R\) is open, \((M\setminus R)\subset M\) is compact. Let \(M_\varepsilon = M\setminus B_\varepsilon(M\setminus R)\), where \(B_\varepsilon(M\setminus R)\) denotes the open \(\varepsilon\)-neighborhood of \(M\setminus R\). The sets \(M_\varepsilon\) then form a compact exhaustion of \(R\) and 
	\[
		\int_M f^\ast\omega\wedge \eta = \int_R f^\ast\omega \wedge \eta= \lim_{\varepsilon \to 0} \int_{M_\varepsilon} f^\ast\omega \wedge \eta \,.
	\]
	\(M_\varepsilon\) can be covered by finitely many open sets \(U_i\subset M_\varepsilon\) such that \(U_i \to f(U_i) =: V_i\) is a \(D\)-fiber bundle, where \(D\) denotes an \((m-n)\)-dimensional disk. As \(M\) and \(N\) are both oriented, each \(U_i\) carries a canonical product orientation. Now the theorem follows by integration over the fiber using a partition of unity \(\{\rho_i\}\) subordinate to \(\{U_i\}\): We have
	\[
		\int_{M_\varepsilon} f^\ast\omega\wedge \eta = \sum_i \int_{U_i} \rho_i\,f^\ast\omega \wedge \eta = \sum_i \int_{V_i} \omega \wedge \fint_{U_i} \rho_i\,\eta = \sum_i \int_{f(M_\varepsilon)\setminus S} \omega \wedge \fint_{U_i} \rho_i\,\eta \,,
	\]
	where \(S\) denotes the null-set of singular values of \(f\) (Sard's lemma). Now, for \(y\in f(M_\varepsilon)\setminus S\),
	\[
		\Bigl(\sum_i\fint_{U_i}\rho_i\,\eta\,\,\Bigr)_y = \sum_i\int_{U_i\cap \Lambda_y}\rho_i\,\eta = \int_{\Lambda_y}\eta = \langle\pcurr_{f,y}| \eta\rangle \,.
	\]
	Now taking the limit for \(\varepsilon\to 0\) we obtain
	\[
		\int_M f^\ast\omega \wedge \eta = \int_R f^\ast\omega \wedge \eta = \int_{f(M)\setminus S} \langle\pcurr_{f}|\eta\rangle\, \omega =  \int_{N} \langle\pcurr_{f}|\eta\rangle\, \omega \,.
	\]
\end{proof}

\bibliographystyle{acm}
\bibliography{randomzerocurrents}

\end{document}